\documentclass[english,11pt,reqno]{amsart}

\usepackage{amssymb}
\usepackage{amsmath}
\usepackage{amsthm}
\usepackage{hyperref}
\usepackage{natbib}

\newcommand{\N}{\mathbb{N}}
\newcommand{\R}{\mathbb{R}}

\newtheorem{theorem}{Theorem}
\newtheorem{proposition}[theorem]{Proposition}

\newtheorem{definition}[theorem]{Definition}
\newtheorem{remark}{Remark}

\def\di{\displaystyle}
\def\N{\mathbb{N}}
\def\R{\mathbb{R}}

\newcommand{\fonction}[5]{\begin{array}[t]{lrcl}#1 :&#2 &\longrightarrow &#3\\&#4& \longmapsto &#5 \end{array}}

\newcommand{\fonctionsansdef}[3]{#1 : #2 \longrightarrow #3}

\begin{document}

\title{Helmholtz theorem for stochastic Hamiltonian systems}
\author{Fr\'ed\'eric Pierret}
\address{SYRTE, Observatoire de Paris, 77 avenue Denfert-Rochereau, 75014 Paris, France}
\begin{abstract}
We derive the Helmholtz theorem for stochastic Hamiltonian systems. Precisely, we give a theorem characterizing Stratonovich stochastic differential equations, admitting a Hamiltonian formulation. Moreover, in the affirmative case, we give the associated Hamiltonian. This result show coherence with the approach of \cite{milstein} where the authors answered to this problem through the preservation of the symplectic form but without providing the Hamiltonian associated.
\end{abstract}

\maketitle
\setcounter{tocdepth}{2}
\tableofcontents

\section{Introduction}

A classical problem in Analysis is the well-known {\it Helmholtz's inverse problem of the calculus of variations}: find a necessary and sufficient condition under which a (system of) differential equation(s) can be written as an Euler--Lagrange or a Hamiltonian equation and, in the affirmative case, find all the possible Lagrangian or Hamiltonian formulations. This condition is usually called \textit{Helmholtz condition}.\\

The Lagrangian Helmholtz problem has been studied and solved by J. Douglas \cite{doug}, A. Mayer \cite{maye} and A. Hirsch \cite{hirs,hirs2}. The Hamiltonian Helmholtz problem has been studied and solved up to our knowledge by R. Santilli in his book \cite{santilli}.\\

Generalization of this problem in the \emph{discrete calculus} of variations framework has been done in \cite{bourdin-cresson} and \cite{hydon1} in the discrete Lagrangian case. For the Hamiltonian case it has been done for the discrete calculus of variations in \cite{opri1} using the framework of \cite{mars} and in \cite{cresson-pierret2} using a discrete embedding procedure derived in \cite{cresson-pierret1}. In the case of \emph{time-scale calculus}, i.e. a mixing between continuous and discrete sub-intervals of time, it has been done in \cite{pierret_helmholtz_ts}.\\

In this paper we generalize the Helmholtz theorem for Hamiltonian systems in the case of \emph{Stratonovich stochastic calculus}. We recover the classical case when there is no stochastic counterpart.\\

The paper is organized as follows: In section 2, we give some generalities and notations about the Stratonovich stochastic calculus. In section 3, we remind definitions and results about classical and stochastic Hamiltonian systems. In section 4, we give a brief survey of the classical Helmholtz Hamiltonian problem and then we prove the main result of this paper, the stochastic Hamiltonian Helmholtz theorem. Finally, in section 5, we conclude and give some prospects.

\section{Generalities and Notations}
In this section we remind the basic definitions and notations of stochastic processes following closely the presentation done in the book of B. {\O}ksendal \cite{oksendal} for which we refer for more details and proofs. \\

In all the paper, we consider $\left(\Omega, \mathcal{F},P\right)$ to be a probability space. \\

\begin{definition}
Let $B(t)$ be a $n$-dimensional Brownian motion. Then we define $\mathcal{F}_t=\mathcal{F}^{(n)}_t$ to be the $\sigma$-algebra generated by the random variable $\left\{B_i(s)\right\}_{1\leq i\leq n, 0\leq s\leq t}$.
\end{definition}

\begin{definition}[Definition 3.3.2 p. 35, \cite{oksendal}]
Let $\mathcal{W}_\mathcal{H}(a,b)$ be the class of functions
	\begin{equation*}
	\fonction{f}{[0,+\infty)\times\Omega}{\R}{(t,\omega)}{f(t,\omega)=f_t(\omega)}
	\end{equation*}
	such that
	\begin{enumerate}
		\item $(t,\omega)\rightarrow f(t,\omega)$ is $\mathcal{B}\times\mathcal{F}$-measurable, where $\mathcal{B}$ denotes the Borel $\sigma$-algebra on $[0,+\infty)$.
		\item There exist a filtration $\mathcal{H}$ on $\left(\Omega, \mathcal{F}\right) $ defining an increasing family of $\sigma$-algebras $\mathcal{H}_{t}, t\geq 0$ such that 
		\begin{enumerate}
			\item  $B_{t}$ is a martingale with respect to $\mathcal{H}_{t}$
			
			\item $f_{t}$ is $\mathcal{H}_{t}$-adapted
		\end{enumerate}
		\item $P\left(\displaystyle \int_{a}^{b}f(s,\omega)^{2}ds<\infty\right)=1$.
	\end{enumerate}

We put $ \di \mathcal{W}_\mathcal{H}=\bigcap_{b>0}\mathcal{W}_\mathcal{H}(0,b)$.
\end{definition}

\begin{definition}
Let $\mathcal{W}^{d\times n}_\mathcal{H}(a,b)$ to be the set of $d\times n$ matrices $v=\left[v_{ij}(t,\omega)\right]$ where each entry $v_{ij}(t,\omega)\in\mathcal{W}_\mathcal{H}(a,b)$. We also put in the same manner as the previous definition, $ \di \mathcal{W}^{d\times n}_\mathcal{H}=\bigcap_{b>0}\mathcal{W}^{d\times n}_\mathcal{H}(0,b)$.
\end{definition}

We now can introduce the set of It\^o processes (see Definition 4.1.1 p. 44 in \cite{oksendal})

\begin{definition}[$1$-dimensional It\^{o} processes]
Let $B(t)$ be a $1$-dimensional Brownian motion on $(\Omega, \mathcal{F}, P)$. A 1-dimensional
\emph{It\^{o} process} or \emph{It\^o integral} is a stochastic process $X(t)$ on $(\Omega, \mathcal{F}, P)$ of the form
\begin{equation}
X(t)=X(0)+\int_{0}^{t}u(s, \omega)ds+\int_{0}^{t}v(s, \omega)dB(s),
\end{equation}
where $v\in \mathcal{W}_{\mathcal{H}}$, $u$ is $\mathcal{H}_{t}$-adapted and is such that
\begin{equation*}
P\left(\displaystyle \int_{0}^{t}|u(s,\omega)|ds<\infty \ \text{for all }t\geq 0\right)=1.
\end{equation*}
In the differential form, $X_t$ is written as
\begin{equation}
dX(t)=uds+vdB_t
\end{equation}
\end{definition}

We now turn in higher dimensions (see \cite{oksendal}, Section 4.2 p. 48)

\begin{definition}[$d$-dimensional It\^{o} processes]
Let $B(t)$ be a $n$-dimensional Brownian motion on $(\Omega, \mathcal{F}, P)$. A $d$-dimensional
It\^{o} process is a stochastic process $X(t)$ on $(\Omega, \mathcal{F}, P)$ of the form
\begin{equation}
X(t)=X(0)+\int_{0}^{t}u(s, \omega)ds+\int_{0}^{t}v(s, \omega) dB(s),
\end{equation}
where $u$ is a $d$-dimensional vector and $v$ is a $d\times n$ matrix such that each entry $u_i(t,\omega)$ and $v_{i,j}((t,\omega))$ defined a 1-dimensional It\^o process, for $i=1,\ldots,d$ as in the previous definition. Generally, when $u(t,\omega)=\mu(t,X(t)$ and $v(t,\omega)=\sigma(t,X(t))$ where $\mu :\R^{d+1} \rightarrow \R^d$, $\sigma :\R^{d+1} \rightarrow \R^{d\times n}$,
\begin{equation}
dX(t)=\mu(t,X(t))dt+\sigma(t,X(t))dB(t)
\end{equation}
is called an \emph{It\^o stochastic differential equation}
\end{definition}

We now define the \emph{Stratonovich integral}  from the \emph{It\^o integral} (see \cite{oksendal}, Paragraph ``A comparison of It\^o and Stratonovich integrals" p. 35 and Section 6.1 p. 85)

\begin{definition}[Stratonovich integral]
\label{definition_strato}
Let $B(t)$ be n-dimensional Brownian motion on $(\Omega, \mathcal{F}, P)$ and let $X(t)$ a $d$-dimensional It\^{o} process such that 
\begin{equation}
X(t)=X(0)+\int_{0}^{t}\left\{\mu(s, X(s))+\tilde{\mu}(s, X(s))\right\}ds+\int_{0}^{t}\sigma(s, X(s)) dB(s),
\end{equation}
where $\mu,\tilde{\mu} :\R^{d+1} \rightarrow \R^d$, $\sigma :\R^{d+1} \rightarrow \R^{d\times n}$, and where for all $1\leq i\leq d$ we have
\begin{equation*}
\tilde{\mu}_{i}(t, x)=\frac{1}{2}\sum_{j=1}^{n}\sum_{k=1}^{d}\frac{\partial \sigma_{ij}}{\partial x_{k}}\sigma_{kj}.
\end{equation*}
The term
\begin{equation*}
\int_{0}^{t}\tilde{\mu}(s, X(s))ds+\int_{0}^{t}\sigma(s, X(s)) dB(s)
\end{equation*}
defines the \emph{Stratonovich integral} (see \cite{oksendal},p.24,2) and it is written as
\begin{equation}
\int_0^t \sigma (s, X(s) ) \circ d B(s) = \int_{0}^{t}\tilde{\mu}(s, X(s))ds+\int_{0}^{t}\sigma(s, X(s)) dB(s).
\end{equation}
Then, the corresponding \emph{Stratonovich process} is written as
\begin{equation}
X(t)=X(0)+\int_{0}^{t}\mu(s, X(s))ds+\int_{0}^{t}\sigma(s, X(s)) \circ dB(s),
\end{equation}
or in differential form as
\begin{equation}
dX(t)=\mu(s, X(s))ds+\sigma(s, X(s)) \circ dB(s).
\end{equation}
which is called a \emph{Stratonovich stochastic differential equation}
\end{definition}

We define the set $\mathcal{S}([a,b],\R^d)$ as the set of $d$-dimensional Stratonovich processes defined in the Definition \ref{definition_strato} over the interval of time $[a,b]$. We formally introduce the set $\mathcal{S}^*([a,b],\R^d)$ as the set of Stratonovich stochastic process in their differential form, i.e. a element $x \in \mathcal{S}^*([a,b],\R^d)$ correspond to a element $X(t)\in \mathcal{S}([a,b],\R^d)$ such that $x=dX(t)$. We also define $\mathcal{S}_0([a,b],\R^d)=\{X \in \mathcal{S}([a,b],\R^d), \ X(a,\omega)=X(b,\omega)=0 \ \text{for all} \ \omega\}$. \\

From Theorem 4.1.5 in \cite{oksendal}, we have the following integration by parts formula for Stratonovich processes
\begin{theorem}[Stratonovich integration by parts formula]
Let $X,Y\in \mathcal{S}([a,b],\R^d)$. We have
\begin{equation}
\int_{a}^{b}X(t) \circ  dY(t)=X(b)\cdot Y(b)-X(a)\cdot Y(a) - \int_{a}^{b}Y(t)\circ dY(t).
\end{equation}
\end{theorem}

We now introduce the definition of \emph{stochastic field}

\begin{definition}
A \emph{stochastic field} $X(Z(t))$ associated to a stochastic process $Z(t)\in\mathcal{S}([a,b],\R^d)$ is the data of its deterministic part and its purely stochastic part as follows: $X(Z(t))=\{\mu(t,Z(t)),\sigma(t,Z(t))\}$.
\end{definition}

In what follow, we omit for notations convenience, when there is no misleading, the explicit dependence in $\omega$ and we put the dependency in time $t$ as a subscript when there is no possible confusion with the coordinates.

\section{Reminder about Hamiltonian systems}

\subsection{Classical Hamiltonian systems}
For simplicity we consider time-independent Hamiltonian. The time-dependent case can be done in the same way.

\begin{definition}[Classical Hamiltonian]
A classical Hamiltonian is a function $H : \R^d \times \R^d \rightarrow \R$ such that for $(q,p)\in C^1([a,b],\R^d) \times C^1([a,b],\R^d)$ we have the time evolution of $(q,p)$ given by the classical Hamilton's equations
\begin{equation}
\left\{
\begin{array}{l l}
\dot{q}&=\frac{\partial H(q,p)}{\partial p} \\
\dot{p}&=-\frac{\partial H(q,p)}{\partial q}
\end{array}
\right.
\label{def_hamilton}
\end{equation}
\end{definition}

A vectorial notation is obtained posing $z=(q,p)^\mathsf{T}$ and $\nabla H = (\frac{\partial H}{\partial q} , \frac{\partial H}{\partial p} )^\mathsf{T}$ where $\mathsf{T}$ denotes the transposition. The Hamilton's equations are then written as
\begin{equation}
\frac{dz}{dt} = J \cdot \nabla H,
\end{equation}
where $J = \begin{pmatrix} 0 & I_d \\ -I_d & 0 \end{pmatrix}$ with $I_d$ the identity matrix on $\R^d$ denotes the symplectic matrix.\\


An important property of Hamiltonian systems is that there solutions correspond to \emph{critical points} of a given functional, i.e. follow from a \emph{variational principle}.

\begin{theorem}
The points $(q,p)\in C^1([a,b],\R^d) \times C^1([a,b],\R^d)$ satisfying Hamilton's equations are critical points of the functional
\begin{equation}
\footnotesize\fonction{\mathcal{L}_H}{C^1([a,b],\R^d)\times C^1([a,b],\R^d)}{\R}{(q,p)}{\mathcal{L}_H(q,p) = \di\int_{a}^{b} \di L_H (q(t),p(t),\dot{q}(t),\dot{p}(t))}
\end{equation}
where $\fonctionsansdef{L_H}{\R^d \times \R^d \times \R^d \times \R^d}{\R}$ is the Lagrangian defined by
\begin{equation*}
\di L_H(x,y,v,w)=y\cdot v-H(x,y).
\end{equation*}
\label{thm_hamiltonian}
\end{theorem}

\subsection{Stochastic Hamiltonian systems}
\label{section-hamiltonian}
Stochastic Hamiltonian systems are formally defined as

\begin{definition}
A stochastic differential equation is called stochastic Hamiltonian system if we can find $\mathbf{H}=\left \{ H_D, H_S \right \}$ with $H_D :\R^{2d} \mapsto \R$ and $H_S :\R^{2d} \mapsto \R^{n}$
such that
\begin{equation}
\left\{
\begin{array}{lll}
dQ & = & \di\frac{\partial H_D}{\partial P} dt + \di\frac{\partial H_S}{\partial P} \circ d B_t \\
dP & = & -\di\frac{\partial H_D}{\partial Q} dt - \di\frac{\partial H_S}{\partial Q} \circ d B_t.
\end{array}
\right.
\end{equation}
\end{definition}

We recover the classical algebraic structure of Hamiltonian systems. The main properties supporting this definition are the following one, already proved in \cite{bismut}\\

\begin{itemize}
\item {\it Liouville's property} Let $(Q,P)\in \R^{2d}$, we consider the stochastic differential equation
\begin{equation}
\label{eq**}
\left\{
\begin{array}{lll}
dQ & = & f(P,Q)dt + \sigma (P,Q) \circ d B_t ,\\
dP & = & g(P,Q)dt + \gamma (P,Q) \circ d B_t .
\end{array}
\right.
\end{equation}
The phase flow of (\ref{eq**}) preserves the {\it symplectic structure} if and only if it is a stochastic Hamiltonian system.	\\
\item {\it Hamilton's principle} Solutions of a stochastic Hamiltonian system correspond to critical points of a stochastic functional defined by 
\begin{equation}
\mathcal{L}_{\mathbf{H}} (Q,P)=\di \int_a^b L_{H_D}(Q(t),P(t) dt + L_{H_S}(Q(t),P(t)\circ d B_t .
\end{equation}
\end{itemize}

\section{Stochastic Helmholtz problem}

In this Section, we solve the inverse problem of the stochastic calculus of variations in the Hamiltonian case. We first recall the usual way to derive the Helmholtz conditions following the presentation made by R. Santilli \cite{santilli}. We have two main derivations 
\begin{itemize}
\item One is related to the characterization of Hamiltonian systems via the \emph{symplectic two-differential form} and the fact that by duality the associated one-differential form to a Hamiltonian vector field is closed. Such conditions are called \emph{integrability conditions}.

\item The second one use the characterization of Hamiltonian systems via the self-adjointness of the Fr\'echet derivative associated to the differential operator associated to the equation. These conditions are usually called \emph{Helmholtz conditions}.
\end{itemize}
Of course, we have coincidence of the two procedures in the classical case. In the stochastic case the first characterization has been studied by \cite{milstein}. As a consequence, we follow the second way to obtain the stochastic analogue of the Helmholtz conditions and we show that there is coherence between the two approach.

\subsection{Hemlholtz conditions for Hamiltonian systems}

\subsubsection{Symplectic scalar product}

In this Section we work on $\R^{2d},d \ge 1, d \in \N$. The \emph{symplectic scalar product} $\langle \cdot,\cdot \rangle_J$ is defined for all $X,Y \in \R^{2d}$ by
\begin{equation}
\langle X,Y\rangle_J = \langle X,JY\rangle ,
\end{equation}
where $\langle \cdot,\cdot \rangle$ denotes the usual scalar product. We also consider the $L^2$ symplectic scalar product induced by $\langle \cdot,\cdot \rangle_J$ defined for $f,g \in C^1([a,b],\R^{2d})$ by 
\begin{equation}
\langle f,g \rangle_{L^2,J}=\di\int_{a}^{b} \langle f(t),g(t) \rangle_Jdt \ .
\end{equation}

\subsubsection{Adjoin of a differential operator}

In the following, we consider first order differential equations of the form
\begin{equation}
\label{equagen}
\frac{d}{dt}\begin{pmatrix}q \\ p \end{pmatrix} = \begin{pmatrix} X_q(q,p) \\ X_p(q,p) \end{pmatrix}.
\end{equation}
The associated differential operator is written as
\begin{equation}
O^{a,b}_X(q,p) = \begin{pmatrix} \dot{q} - X_q(q,p) \\ \dot{p} - X_p(q,p) \end{pmatrix} \ . 
\label{operatorO}
\end{equation}

A \emph{natural} notion of adjoin for a differential operator is then defined.

\begin{definition}
Let $\fonctionsansdef{A}{C^1([a,b],\R^{2n})}{C^1([a,b],\R^{2n})}$. We define the adjoin $A^*_J$ of $A$ with respect to $<\cdot,\cdot>_{L^2,J}$ by
\begin{equation}
<A \cdot f, g>_{L^2,J} = <A^{*}_J \cdot g, f>_{L^2,J} \ .
\end{equation}
\end{definition}

An operator $A$ will be called \emph{self-adjoin} if $A=A^{*}_J$ with respect to the $L^2$ symplectic scalar product.

\subsubsection{Hamiltonian Helmholtz conditions}

The Helmholtz's conditions in the Hamiltonian case are given by (see Theorem. 3.12.1, p.176-177 in \cite{santilli})

\begin{theorem}[Hamiltonian Helmholtz theorem]
Let $X(q,p)$ be a vector field defined by $X(q,p)^\mathsf{T}= (X_q(q,p) , X_p(q,p) )$. The differential equation (\ref{equagen}) is Hamiltonian if and only the associated differential operator $O^{a,b}_X$ given by (\ref{operatorO}) has a  self adjoin Fr\'echet derivative with respect to the symplectic scalar product. 

In this case the Hamiltonian is given by 
\begin{equation}
H(q,p)=\int_{0}^{1}\left[p \cdot X_q(\lambda q, \lambda p) - q\cdot X_p(\lambda q, \lambda p) \right]d\lambda
\end{equation}
\end{theorem}

The conditions for the self-adjointness of the differential operator can be made \emph{explicitly}. They coincide with the \emph{integrability conditions} characterizing the exactness of the one-form associated to the vector field by duality (see \cite{santilli}, Theorem.2.7.3 p.88).

\begin{theorem}[Integrability conditions]
Let $X(q,p)^\mathsf{T}= (X_q(q,p) , X_p(q,p) )$ be a vector field. The differential operator $O^{a,b}_X$ given by (\ref{operatorO}) has a  self adjoin Fr\'echet derivative with respect to the symplectic scalar product if and only if 
\begin{equation}
\frac{\partial X_q}{\partial q} + \left(\frac{\partial X_p}{\partial p} \right)^\mathsf{T} = 0, \quad \frac{\partial X_q}{\partial p} \ \text{and} \ \frac{\partial X_p}{\partial q} \ \text{are symmetric} .
\end{equation}
\end{theorem}

Of course, the first condition corresponds to the fact that Hamiltonian systems are \emph{divergence free}, i.e. we have $\mbox{div} X =0$. 

\subsection{Stochastic Helmholtz's conditions}

We derive the main result of the paper giving the characterization of stochastic differential equations which are stochastic Hamiltonian systems. 

\subsubsection{Stochastic symplectic scalar product}

Following the classical case, we introduce a stochastic analogue of the symplectic scalar product and the notion of self-adjointness for stochastic differential equations. \\

Consider two stochastic process $X_t, Y_t \in \mathcal{S}([a,b],\R^{2d})$ written as $dX_t = A(t,X_t)dt+B(t,X_t) \circ d B_t$ and $dY_t = C(t,Y_t)dt+D(t,Y_t) \circ d B_t $. We define

\begin{definition}[$L^2$-Stratonovich scalar product]
	We define the \\ $L^2$-Stratonovich scalar product of $X_t$ and $dY_t$ as
	\begin{equation}
	<X_t,dY_t>_{L^2,\text{Strato}}=\int_{a}^{b} X_t \circ dY_t
	\end{equation}
\end{definition}

The symplectic version is obtained as in the classical case.

\begin{definition}[$L^2$-Stratonovich symplectic scalar product]
	 We define the $L^2$-Stratonovich symplectic scalar product of $X_t$ and $dY_t$ such as
	\begin{equation}
	<X_t,dY_t>_{L^2,J,\text{Strato}} = <X_t,J \cdot dY_t>_{L^2,\text{Strato}}
	\end{equation}
\end{definition}

\subsubsection{Adjoin of a stochastic differential operator}

We consider the stochastic differential equations of the form
\begin{equation}
\label{equagens}
\begin{pmatrix} dQ \\ dP \end{pmatrix} = \begin{pmatrix} X_{Q,D}dt + X_{Q,S} \circ d B_t\\ X_{P,D}dt + X_{P,S} \circ d B_t \end{pmatrix}.
\end{equation}

The stochastic differential equations \eqref{equagens} is associated with the stochastic field $X(Q,P)=\{X_D(Q,P),X_S(Q,P)\}$. Its associated stochastic differential operator is written as
\begin{equation}
O_X(Q,P) = \begin{pmatrix} dQ - X_{Q,D}dt - X_{Q,S} \circ d B_t \\ dP - X_{P,D}dt - X_{P,S} \circ d B_t \end{pmatrix} \ . 
\label{operatorOs}
\end{equation}

A notion of {\it symplectic adjoin} for stochastic differential operators can be defined

\begin{definition}
Let $\fonctionsansdef{A}{\mathcal{S}([a,b],\R^{2d})}{\mathcal{S}^*([a,b],\R^{2d})}$.  We define the adjoin $\fonctionsansdef{A^*_J}{\mathcal{S}([a,b],\R^{2d})}{\mathcal{S}^*([a,b],\R^{2d})}$ of $A$ with respect to $<\cdot,\cdot>_{L^2,J,\text{Strato}}$ by
	\begin{equation}
	<A \cdot X_t, dY_t>_{L^2,J,Strato} = <A^{*}_J \cdot dY_t, X_t>_{L^2,J} \
	\end{equation}
\label{adjointstoc}
\end{definition}

As in the classical case, a main role will be played by self-adjoin stochastic differential operators.

\begin{definition}
A stochastic differential operator $A$ is said to be self-adjoin with respect to the symplectic scalar product $<\cdot,\cdot>_{L^2,J,\text{Strato}}$ if $A=A^{*}_J $.
\label{selfadjointsto}
\end{definition}

\subsubsection{Stochastic Helmholtz conditions}
 The main result of this Section is the stochastic analogue of the Hamiltonian Helmholtz conditions for stochastic differential equations. \\

\begin{proposition}
\label{main2}
Let $U,V \in \mathcal{S}_0([a,b],\R^{d})$. The Fr\'echet derivative $DO(Q,P)$ of (\ref{operatorOs}) is given by
\begin{align}
& DO(Q,P)(U,V) = \nonumber\\
&\begin{pmatrix} d U -\left[\frac{\partial X_{Q,D}}{\partial Q} \cdot U +\frac{\partial X_{Q,D} }{\partial P}\cdot V\right]dt - \left[\frac{\partial X_{Q,S}}{\partial Q} \cdot U +\frac{\partial X_{Q,S} }{\partial P}\cdot V\right]\circ d B_t \\ d V -\left[\frac{\partial X_{P,D}}{\partial Q} \cdot U +\frac{\partial X_{P,D} }{\partial P}\cdot V\right]dt - \left[\frac{\partial X_{P,S}}{\partial Q} \cdot U +\frac{\partial X_{P,S} }{\partial P}\cdot V\right]\circ d B_t  \end{pmatrix}
\end{align}
and his adjoin ${DO^*_J}(Q,P)$ with respect to the symplectic scalar product ${<\cdot,\cdot>_{L^2,J,\text{Strato}}}$ is given by
{\footnotesize\begin{align}
& {DO_J^*}(Q,P)(U,V) =\nonumber \\
&\begin{pmatrix} d U -\left[-\left(\frac{\partial X_{P,D}}{\partial P}\right)^\mathsf{T} \cdot U +\left(\frac{\partial X_{Q,D} }{\partial P}\right)^\mathsf{T}\cdot V\right]dt - \left[-\left(\frac{\partial X_{P,S}}{\partial P}\right)^\mathcal{T} \cdot U +\left(\frac{\partial X_{Q,S} }{\partial P}\right)^\mathcal{T}  \cdot V\right]\circ d B_t \\ d V -\left[\left(\frac{\partial X_{P,D}}{\partial Q}\right)^\mathsf{T} \cdot U -\left(\frac{\partial X_{Q,D} }{\partial Q}\right)^\mathsf{T}\cdot V\right]dt -  \left[\left(\frac{\partial X_{P,S}}{\partial Q}\right)^\mathcal{T} \cdot U -\left(\frac{\partial X_{Q,S} }{\partial Q}\right)^\mathcal{T}  \cdot V\right]\circ d B_t  \end{pmatrix}
\end{align}}
\end{proposition}

\begin{remark}
The notation ${}^\mathcal{T}$ correspond to a transposition as follows: \\

As $X_{Q,S}$ and $X_{P,S}$ are in $\R^{d\times n}$ then each partial derivative with respect to $Q$ or $P$ are in $\R^{d\times n \times d}$. For example, $\frac{\partial X_{Q,S}}{\partial Q}=\left( \frac{\partial \left(X_{Q,S}\right)_{i,j}}{\partial Q_k}\right)_{1\leq i\leq d,\ 1\leq j\leq n, \ 1\leq k \leq d}$ and then $\left(\frac{\partial X_{Q,S}}{\partial Q}\right)^\mathcal{T}=\left( \frac{\partial \left(X_{Q,S}\right)_{k,j}}{\partial Q_i}\right)_{1\leq i\leq d,\ 1\leq j\leq n, \ 1\leq k \leq d}$ with respect to the same indices.
\end{remark}

\begin{proof}
Let $U,V \in \mathcal{S}_0([a,b],\R^{d})$ and $F,G \in \mathcal{S}([a,b],\R^{d})$. The Fr\'echet derivative $DO(Q,P)$ follows from simple computations and is given by 
{\footnotesize\begin{align}
&\langle DO(Q,P)(U,V),(F,G) \rangle_{L^2,J,\text{Strato}}= \nonumber \\
&\int_{a}^{b}\bigg( d U \cdot G -\left[\frac{\partial X_{Q,D}}{\partial Q} \cdot U +\frac{\partial X_{Q,D} }{\partial P}\cdot V\right]\cdot G \ dt - \left\{\left[\frac{\partial X_{Q,S}}{\partial Q} \cdot U +\frac{\partial X_{Q,S} }{\partial P}\cdot V\right]\circ d B_t\right\}\cdot G \nonumber \\
 &- d V \cdot F +\left[\frac{\partial X_{P,D}}{\partial Q} \cdot U +\frac{\partial X_{P,D} }{\partial P}\cdot V\right]\cdot F \ dt + \left\{\left[\frac{\partial X_{P,S}}{\partial Q} \cdot U +\frac{\partial X_{P,S} }{\partial P}\cdot V\right]\circ d B_t\right\}\cdot F \bigg)  
\end{align}}
Using the Stratonovich integration by parts formula, we obtain 
{\footnotesize
\begin{align*}
&\langle DO(Q,P)(U,V),(F,G) \rangle_{L^2,J,\text{Strato}}= \nonumber \\
&\int_{a}^{b}\bigg( - U \cdot dG -\left[\left( \left(\frac{\partial X_{Q,D}}{\partial Q}\right)^\mathsf{T} \cdot G\right)\cdot U +\left(\left(\frac{\partial X_{Q,D} }{\partial P}\right)^\mathsf{T}\cdot G\right)\cdot V\right] \ dt \nonumber\\
&- \left\{\left[\left(\frac{\partial X_{Q,S}}{\partial Q}\right)^\mathcal{T} \cdot G \right]\circ d B_t\right\}\cdot U - \left\{\left[\left(\frac{\partial X_{Q,S} }{\partial P}\right)\cdot G\right]\circ d B_t\right\}\cdot V\nonumber \\
&-V \cdot dF +\left[\left( \left(\frac{\partial X_{P,D}}{\partial Q}\right)^\mathsf{T} \cdot F\right)\cdot U +\left(\left(\frac{\partial X_{P,D} }{\partial P}\right)^\mathsf{T}\cdot F\right)\cdot V\right] \ dt \nonumber\\
&+ \left\{\left[\left(\frac{\partial X_{P,S}}{\partial Q}\right)^\mathcal{T} \cdot F \right]\circ d B_t\right\}\cdot U - \left\{\left[\left(\frac{\partial X_{P,S} }{\partial P}\right)\cdot F\right]\circ d B_t\right\}\cdot V \bigg)  
\end{align*}}
As a consequence, the symplectic stochastic adjoin of $DO(Q,P)$ is given by
{\footnotesize\begin{align*}
& {DO_J^*}(Q,P)(F,G) = \\
&\begin{pmatrix} d F -\left[-\left(\frac{\partial X_{P,D}}{\partial P}\right)^\mathsf{T} \cdot F +\left(\frac{\partial X_{Q,D} }{\partial P}\right)^\mathsf{T}\cdot G\right]dt - \left[-\left(\frac{\partial X_{P,S}}{\partial P}\right)^\mathcal{T} \cdot F +\left(\frac{\partial X_{Q,S} }{\partial P}\right)^\mathcal{T}  \cdot G\right]\circ d B_t \\ d G -\left[\left(\frac{\partial X_{P,D}}{\partial Q}\right)^\mathsf{T} \cdot F -\left(\frac{\partial X_{Q,D} }{\partial Q}\right)^\mathsf{T}\cdot G\right]dt -  \left[\left(\frac{\partial X_{P,S}}{\partial Q}\right)^\mathcal{T} \cdot F -\left(\frac{\partial X_{Q,S} }{\partial Q}\right)^\mathcal{T}  \cdot G\right]\circ d B_t  \end{pmatrix}
\end{align*}}
which concludes the proof.
\end{proof}

It follows directly, an explicit characterization of stochastic vector fields satisfying the stochastic Hamiltonian Helmholtz conditions. By coherence with the classical case, we call them \emph{stochastic integrability conditions}.

\begin{proposition}[Stochastic integrability conditions]
The operator $O_X$ defined by (\ref{operatorOs}) has a self-adjoin Fr\'echet derivative at $(Q,P)\in \mathcal{S}([a,b],\R)\times \mathcal{S}([a,b],\R)$ if and only if the conditions
\begin{align}
&\frac{\partial X_{Q,D}}{\partial Q} + \left(\frac{\partial X_{P,D}}{\partial P} \right)^\mathsf{T} = 0  , \\
&\frac{\partial X_{Q,D}}{\partial P}= \left(\frac{\partial X_{Q,D}}{\partial P}\right)^\mathsf{T} \ \text{and} \quad \frac{\partial X_{P,D}}{\partial Q}= \left( \frac{\partial X_{P,D}}{\partial Q}\right)^\mathsf{T}, \\
&\frac{\partial X_{Q,S}}{\partial Q} + \left(\frac{\partial X_{P,S}}{\partial P} \right)^\mathcal{T} = 0  , \\
&\frac{\partial X_{Q,S}}{\partial P}= \left(\frac{\partial X_{Q,S}}{\partial P}\right)^\mathcal{T} \ \text{and} \quad \frac{\partial X_{P,S}}{\partial Q} = \left( \frac{\partial X_{P,S}}{\partial Q}\right)^\mathcal{T},
\end{align}
are satisfied over $[a,b]$.
\end{proposition}

\begin{remark}
Theses conditions are the same that the authors in \cite{milstein} obtained. It means, as in the classical case, there is coherence between the conservation of the symplectic form associated with the stochastic field $X(Q,P)$ and the self-adjointness of the stochastic differential operator $O_X$.
\end{remark}

\begin{theorem}[Stochastic Hamiltonian Helmholtz theorem]
\label{main1}
Let $X(Q,P)=\{X_D(Q,P),X_S(Q,P)\}$ be a stochastic field. The stochastic differential equations (\ref{equagens}) is a stochastic Hamiltonian equation if and only if the operator $O_X$ defined by (\ref{operatorOs}) has a self-adjoin Fr\'echet derivative with respect to the symplectic scalar product $<\cdot,\cdot>_{L^2,J,\text{Strato}}$. \\

\noindent Moreover, in this case, the Hamiltonian $\mathbf{H}=\{H_D, H_S\}$ is given by 
\begin{align}
\label{ham_form}
H_D(Q,P)&=\int_{0}^{1}\left[P \cdot X_{Q,D}(\lambda Q, \lambda P) - Q\cdot X_{P,D}(\lambda Q, \lambda P) \right]d\lambda, \\
H_S(Q,P)&=\int_{0}^{1}\left[P \cdot X_{Q,S}(\lambda Q, \lambda P) - Q\cdot X_{P,S}(\lambda Q, \lambda P) \right]d\lambda,
\end{align}
\end{theorem}

\begin{proof}
If $X$ is a stochastic Hamiltonian field then there exist a function $H_D \in C^2(\R^d \times \R^d , \R )$ and $H_S \in C^2(\R^d \times \R^d , \R^n )$ such that $X_{Q,D}=\frac{\partial H_D}{\partial P}, X_{P,D}=-\frac{\partial H_D}{\partial Q}$ and$X_{Q,S}=\frac{\partial H_S}{\partial P}, X_{P,S}=-\frac{\partial H_S}{\partial Q}$ . The stochastic integrability conditions are satisfied by the Schwarz lemma. \\

Reciprocally, we suppose that $X$ satisfies the stochastic integrability conditions. First, we compute $\frac{\partial H_D}{\partial Q}$. We denote $(\ast)=(\lambda Q, \lambda P)$ and we have
\begin{equation}
\frac{\partial H_D}{\partial Q}(Q,P) = \int_{0}^{1} \bigg( \lambda P \cdot \left(\frac{\partial X_{Q,D}(\ast)}{\partial Q}\right)^\mathsf{T} - X_{Q,D}-\lambda Q \cdot \left(\frac{\partial X_{P,D}(\ast)}{\partial Q}\right)^\mathsf{T}\bigg)d\lambda \nonumber
\end{equation}
Using the stochastic integrability conditions for the deterministic part, we obtain
\begin{equation}
\frac{\partial H_D}{\partial Q}(Q,P) = \int_{0}^{1} \bigg( -\lambda P \cdot \frac{\partial X_{P,D}(\ast)}{\partial P}- X_{Q,D}-\lambda Q \cdot \frac{\partial X_{P,D}(\ast)}{\partial Q}\bigg)d\lambda \nonumber
\end{equation}
Remarking that
\begin{equation*}
\frac{\partial X_{P,D}(\ast)}{\partial \lambda} = \frac{\partial X_{P,D}(\ast)}{\partial P} \cdot P + \frac{\partial X_{P,D}(\ast)}{\partial Q}\cdot Q,
\end{equation*}
we deduce
\begin{equation*}
\frac{\partial H_D}{\partial Q}(Q,P) = \int_{0}^{1} - \frac{\partial}{\partial \lambda}\left( \lambda X_{P,D}(\ast)\right)d\lambda \nonumber
\end{equation*}
and finally, integrating with respect to $\lambda$
\begin{equation*}
\frac{\partial H_D}{\partial Q}(Q,P) = - X_{P,D}(Q,P)
\end{equation*}
In the same way we obtain
\begin{equation*}
\frac{\partial H_D}{\partial P}(Q,P) = X_{Q,D}(Q,P).
\end{equation*}
Second, we compute $\frac{\partial H_S}{\partial Q}$. We have
\begin{equation*}
\frac{\partial H_S}{\partial Q}(Q,P) = \int_{0}^{1} \bigg( \lambda P \cdot \left(\frac{\partial X_{Q,S}(\ast)}{\partial Q}\right)^\mathcal{T} - X_{Q,S}-\lambda Q \cdot \left(\frac{\partial X_{P,S}(\ast)}{\partial Q}\right)^\mathcal{T}\bigg)d\lambda \nonumber
\end{equation*}
Using the stochastic integrability conditions for the purely stochastic part, we obtain using the same trick as previous
\begin{equation}
\frac{\partial H_S}{\partial Q}(Q,P) = - X_{P,S}(Q,P)
\end{equation}
and
\begin{equation}
\frac{\partial H_S}{\partial P}(Q,P) = X_{Q,S}(Q,P)
\end{equation}
This concludes the proof.
\end{proof}

\begin{remark}
The Stratonovich calculus seems to be more appropriate for such computations and also for its result about the variational structure developed by J-M Bismut in \cite{bismut} contrary to the It\^o calculus. However, one can define a notion of It\^o Hamiltonian systems following the self-adjointness characterization. Indeed, the definition of Stratonovich symplectic scalar product can be also defined with the It\^o integral. But one will have a counterpart in the integration by parts in the proof of Proposition \ref{main2}. It leads to an additional set of conditions on the deterministic part of the stochastic field due to the extra term appearing with the It\^o formula which can be very restrictive. 

Also, in the Stratonovich calculus, the Hamiltonian defined is a \emph{first integral} over the solutions of the stochastic differential equations considered, i.e. a constant process, as in the classical case. Whereas the definition obtained with the It\^o calculs as explained previously leads to a second extra set of conditions in order to have the Hamiltonian defined as a first integral. It leads to a real need of interpretation of Hamiltonian systems for stochastic calculus.
\end{remark}

\section{Conclusion and prospects}

We proved a result on Stratonovich stochastic differential equations which allows us to find the existence of a Hamiltonian structure associated and in the affirmative case to give the Hamiltonian. Our result cover the classical case when there is no purely stochastic counterpart.

An important extension of this result concern the stochastic time-scale calculus and more precisely the stochastic calculus on time scales developed in \cite{sanyal} and \cite{bohnersto}. First, the work is to define a notion of Stratonovich calculus on time-scale as it has been done for It\^o calculus in \cite{sanyal} and \cite{bohnersto}. Second, the work is to define a natural notion of stochastic Hamiltonian on time scales and then to give the stochastic time-scale version of Theorem \ref{main1}. This extension is a work in progress and will be the subject of a future paper.


\end{document}